\documentclass[11pt,a4paper,reqno]{amsart}
\usepackage[utf8]{inputenc}
\usepackage{amsmath,amssymb,amsthm}
\usepackage{geometry}
\usepackage{hyperref}
\hypersetup{
    colorlinks=true,
    linkcolor=blue,
    citecolor=blue,
    urlcolor=blue
}

\newtheorem{theorem}{Theorem}[section]
\newtheorem{theo}[theorem]{Theorem}         

\newtheorem{lemm}[theorem]{Lemma}           

\newtheorem{prop}[theorem]{Proposition}     

\theoremstyle{definition}

\theoremstyle{remark}

\newtheorem{rema}[theorem]{Remark}          

\numberwithin{equation}{section}

\title[Totally geodesic submanifolds of a complex flag]{Totally geodesic submanifolds of a complex flag}

\author[]{DANIELLE VELLOSO FERREIRA}
\address{Instituto de Matemática e Estatística, Universidade de São Paulo, Rua do Matão 1010, São Paulo, SP 05508-090, Brazil}
\email{dani.velloso@ime.usp.br}

\thanks{The author has been supported by FAPESP (grant 2025/00281-7)}

\begin{document}

\begin{abstract}
This work presents a classification of the maximal totally geodesic submanifolds of the complex manifold
\[
M = SU(n+2)/S(U(1)\times U(1)\times U(n)) \cong F_{1,2}(\mathbb{C}^{n+2}),\qquad n\geq 2,
\]
equipped with its standard homogeneous Riemannian metric. 
Using the natural fibration \( S^2 \to M \to G_2(\mathbb{C}^{n+2}) \) and curvature computations derived from O’Neill’s formulas, 
we identify curvature-invariant subspaces of the tangent representation. 
As a consequence, we obtain a complete description of the maximal totally geodesic submanifolds of \( M \). 
These include twistor spaces of lower-dimensional complex Grassmannians, real flag manifolds, products of complex projective spaces, 
and twistor spaces of quaternionic projective spaces.
\end{abstract}

\maketitle


\section{Introduction}

The main result of this article is the following.

\begin{theo}\label{Teorema}
Let $n\ge2$, and let $\Sigma$ be a maximal connected complete totally geodesic
submanifold of the partial complex flag manifold
\begin{equation*}
M=
\frac{SU(n+2)}{S(U(1)\times U(1)\times U(n))}
=
F_{1,2}(\mathbb C^{n+2}),
\end{equation*}
endowed with the normal homogeneous Riemannian metric.
Then $\Sigma$ is congruent to one of the following submanifolds:
\begin{itemize}
    \item a lower-dimensional partial complex flag manifold,
    $F_{1,2}(\mathbb C^{n+1})$;

    \item a partial real flag manifold,
    $F_{1,2}(\mathbb R^{n+2})$;

    \item a product of complex projective spaces,
    $\mathbb CP^p\times\mathbb CP^q$, with $p+q=n$;

    \item the twistor space of the quaternionic projective space
    $\mathbb HP^k$, namely $\mathbb CP^{2k+1}$,
    where $k=\lfloor n/2\rfloor$.
\end{itemize}

Moreover, all these submanifolds are reflective and well-positioned with
respect to the natural submersion to the complex Grassmannian of
$2$-planes~\eqref{submersão}.
\end{theo}

Our work extends to the complex setting the classification obtained in
\cite{GKR} for real Stiefel manifolds and generalizes, to arbitrary
dimension, part of the classification established in \cite{RN} for
$F(\mathbb C^3)$ and $\mathbb CP^3$.

The proof combines the geometry of the twistor fibration with a detailed
analysis of the Jacobi operators associated with singular tangent vectors with respect to the isotropy representation of the complex Grassmannian.
After diagonalizing these operators, we exploit the invariance of the
tangent spaces under both the curvature tensor and its first covariant
derivative to determine all possible maximal totally geodesic
submanifolds.

The existence of totally geodesic submanifolds is closely related to the amount of symmetry of the ambient manifold. For this reason, they have been studied extensively in highly symmetric spaces, particularly in the settings of Riemannian symmetric spaces and homogeneous spaces, where the ambient geometry provides powerful tools for their construction and classification.

The study of totally geodesic submanifolds dates back to Hadamard, who introduced the notion at the beginning of the twentieth century. The subject was later placed on a firm theoretical foundation by Cartan through his existence theorem for totally geodesic submanifolds in real analytic manifolds. In the setting of Riemannian symmetric spaces, the classification program began with Wolf's work on rank-one symmetric spaces and was substantially extended by Chen and Nagano, who classified maximal totally geodesic submanifolds in compact symmetric spaces.

For symmetric spaces, totally geodesic submanifolds are characterized by an algebraic condition: their tangent spaces are Lie triple systems. Although this provides a complete algebraic characterization, applying this criterion in practice is often highly nontrivial, and the classification problem remains difficult except for particular families of symmetric spaces.

Beyond the symmetric setting, however, no analogous algebraic characterization is available. Tojo's criterion extends the Lie triple system characterization to naturally reductive homogeneous spaces, where totally geodesic submanifolds are characterized by the interaction between the curvature tensor and the connection.

\begin{theo}[{\cite{T}} -- Tojo's Criterion]\label{Tojo}
A subspace $\mathfrak{s}\subset\mathfrak{m}$ is tangent to a totally
geodesic submanifold if and only if
$e^{\nabla X}(\mathfrak{s})$
is $R$-invariant for every $X\in\mathfrak{s}$.
\end{theo}

Although general, Tojo's criterion is not always straightforward to
apply, since it involves simultaneously the curvature tensor and the
connection. In the present setting, the natural twistor fibration
allows us to exploit additional geometric structure and reduce the
problem to a finite-dimensional algebraic analysis.

Along the proof of the Theorem~\ref{Teorema}, we obtain two
further classification results. The first concerns the twistor spaces of $\mathbb{H}P^k$,
$\mathbb CP^{2k+1}$.

\begin{theo}\label{tg.cp2k+1}
Let $k\geq2$ and let $\Sigma$ be a maximal connected complete totally
geodesic submanifold of
\[
\mathbb CP^{2k+1}
=
\frac{Sp(k+1)}{S^1\times Sp(k)}
\]
with the normal homogeneous Riemannian metric.
Then $\Sigma$ is congruent to one of the following submanifolds:
\begin{itemize}
    \item a lower-dimensional twistor space of $\mathbb HP^k$,
    namely $\mathbb CP^{2k-1}$;

    \item a complex projective space $\mathbb CP^k$;

    \item a Berger sphere quotient
    $\mathbb RP^{2k+1}_{\mathbb C,1/2}(\sqrt2)$.
\end{itemize}
Moreover, all these submanifolds are reflective and well-positioned with
respect to the submersion~\eqref{submersão}.
\end{theo}

The case $k=1$ was completely classified by
Lorenzo-Naveiro and Rodríguez-Vázquez~\cite{RN}.

The second result concerns the real partial flag manifolds.

\begin{theo}\label{real.flag}
Let $n\geq2$ and let $\Sigma$ be a maximal connected complete totally
geodesic submanifold of
\[
F_{1,2}(\mathbb R^{n+2})
=
\frac{SO(n+2)}{S(O(1)\times O(1)\times O(n))}.
\]
Then $\Sigma$ is congruent to one of the following submanifolds:
\begin{itemize}
    \item a lower-dimensional real partial flag manifold
    $F_{1,2}(\mathbb R^{n+1})$;

    \item a product
    $\mathbb RP^p\times\mathbb RP^q$,
    where $p+q=n$;

    \item a Berger sphere quotient
    $\mathbb RP^{2k+1}_{\mathbb C,1/2}(\sqrt2)$,
    where $k=\lfloor n/2\rfloor$.
\end{itemize}
Moreover, all these submanifolds are well-positioned with respect to
the submersion~\eqref{submersão}.
\end{theo}

\subsection{Organization of the paper}

The proof of Theorem~\ref{Teorema} is based on the study of the tangent space
$\mathfrak{s}=T_o\Sigma$ of a totally geodesic submanifold at the base point.
Since $M$ is naturally reductive, Tojo's criterion reduces the problem to
understanding the invariance of $\mathfrak{s}$ under the curvature tensor and
its covariant derivatives. Our strategy consists of choosing suitable singular
vectors with respect to the isotropy representation of the complex Grassmannian in $\mathfrak{s}$, diagonalizing the corresponding Jacobi operator
$R_X$, and exploiting the resulting eigenspace decomposition together with the
$R$- and $\nabla R$-invariance of $\mathfrak{s}$.

Section~\ref{prelim} introduces the geometric background on the partial flag
manifold $F_{1,2}(\mathbb{C}^{n+2})$, its twistor fibration, and the notation
used throughout the paper. In Section~\ref{diag}, we compute the eigenspace
decomposition of the Jacobi operator associated with singular vectors, while
Section~\ref{singular} proves that every totally geodesic submanifold with
nontrivial horizontal component contains a singular vector.

The classification of well-positioned totally geodesic submanifolds is carried
out in Section~\ref{classif}. According to the type of singular vector contained
in $\mathfrak{s}$, the analysis splits into several cases, each leading to one
of the families appearing in Theorem~\ref{Teorema}. Section~\ref{nbp} is devoted
to the non well-positioned case, where we prove that no additional maximal
examples occur for $n\ge2$.

Finally, Section~\ref{further} presents two additional classification results.
The first follows directly from our classification theorem and describes the
maximal connected complete totally geodesic submanifolds of the twistor spaces
$\mathbb{C}P^{2k+1}$. The second combines our methods with the classification
of maximal totally geodesic submanifolds of real Stiefel manifolds obtained
in~\cite{GKR} to classify the maximal connected complete totally geodesic
submanifolds of the real partial flag manifolds
$F_{1,2}(\mathbb{R}^{n+2})$.

The results presented in this paper are part of the author's PhD dissertation.

\section{The geometry of the flag}\label{prelim}

\subsection{The twistor fibration}

The complex flag manifold $M=F_{1,2}(\mathbb{C}^{n+2})$ is the twistor space of the complex Grassmannian $G_2(\mathbb{C}^{n+2})$. The corresponding twistor fibration is given by
\begin{equation}\label{submersão}
    K/H\longrightarrow G/H \stackrel{\pi}{\longrightarrow} G/K,
\end{equation}
where $G=SU(n+2)$, $K=S(U(2)\times U(n))$, and $H=S(U(1)\times U(1)\times U(n))$.

We now describe this fibration in homogeneous terms. The Grassmannian $G_2(\mathbb{C}^{n+2})=G/K$ is a symmetric space with Cartan decomposition $\mathfrak{g}=\mathfrak{k}\oplus\mathfrak{p}$, where $\mathfrak{g}=\mathfrak{su}_{n+2}$, $\mathfrak{k}=\mathfrak{s}(\mathfrak{u}_2\oplus\mathfrak{u}_n)$, and $\mathfrak{p}=\mathbb{C}^n\otimes\mathbb{C}^2$. Accordingly, we identify
\begin{equation*}
\mathfrak{k}=
\left\{
X=
\begin{pmatrix}
P & 0\\
0 & Q
\end{pmatrix}
:
P\in\mathfrak{u}_2,\,
Q\in\mathfrak{u}_n,\,
\operatorname{tr}(X)=0
\right\} \mbox{ and }
\mathfrak{p}=
\left\{
\begin{pmatrix}
0 & Z\\
-Z^* & 0
\end{pmatrix}
:
Z\in M_{2\times n}(\mathbb{C})
\right\}.
\end{equation*}

Let $\mathfrak{g}=\mathfrak{h}\oplus\mathfrak{m}$ be the naturally reductive decomposition of $M$, where $\mathfrak{h}=\mathfrak{s}(\mathfrak{u}_1\oplus\mathfrak{u}_1\oplus\mathfrak{u}_n)$ and $\mathfrak{m}=\mathfrak{z}\oplus\mathfrak{p}$. We refer to $\mathfrak{z}$ and $\mathfrak{p}$ as the vertical and horizontal components of $\mathfrak{m}$, respectively, with respect to the twistor fibration.

We say that a submanifold $\Sigma$ is \textit{well-positioned} with respect to a fibration if
\begin{equation*}
    T_p\Sigma=(T_p\Sigma\cap\mathcal{H}_p)\oplus(T_p\Sigma\cap\mathcal{V}_p)
\end{equation*}
for every $p\in\Sigma$, where $\mathcal{H}$ and $\mathcal{V}$ denote the horizontal and vertical distributions, respectively.

\subsection{Complex structures on \texorpdfstring{$M$}{M}}

Since $G_2(\mathbb{C}^{n+2})$ is a quaternion-Kähler symmetric space, the tangent space $\mathfrak{p}=T_oG_2(\mathbb{C}^{n+2})$ carries a natural quaternionic structure, where $o\in G_2(\mathbb{C}^{n+2})$ is a fixed base point. To describe this structure explicitly, let $E_{ij}$ denote the $(n+2)\times(n+2)$ matrix whose $(i,j)$-entry is $1$ and whose remaining entries are $0$, and define
$X_{ij}=E_{ij}-E_{ji}$ and $Y_{ij}=\sqrt{-1}(E_{ij}+E_{ji})$.

Let $V=X_{12}$ and $W=Y_{12}$ be an orthonormal basis of $\mathfrak{z}$. This basis determines three anti-commuting complex structures on $\mathfrak{p}$, giving rise to a quaternionic structure:
\[
J_VX=[X,V],\qquad
J_WX=[X,W],\qquad
J_{VW}X=J_V(J_WX),
\qquad X\in\mathfrak{p}.
\]
The existence of these three complex structures is independent of the choice of the orthonormal basis $(V,W)$ of $\mathfrak{z}$. Indeed, the same construction can be carried out for any orthonormal pair $(Z,Z^\perp)$ in $\mathfrak{z}$.

In addition to the quaternionic structure above, we consider the complex structure $J_I X=[X,I]$ on $\mathfrak{p}$, where $I=\sqrt{-1}(E_{11}+E_{22})$. This complex structure extends naturally to $\mathfrak{m}$ by
\[
J=J_I|_{\mathfrak{p}}+r|_{\mathfrak{z}},
\]
where $r$ denotes the rotation of angle $\pi/2$ on $\mathfrak{z}$ satisfying $r(V)=-W$. Then $J$ defines an integrable complex structure on $M$ \cite{S}. We also define the almost complex structure
\[
\tilde{J}=J_I|_{\mathfrak{p}}-r|_{\mathfrak{z}},
\]
which is not integrable.

\subsection{Curvature tensor}

Using the twistor fibration together with O'Neill's formulas, we derive the curvature tensor of $M$. More precisely, for a fixed $X\in\mathfrak{p}$, we compute the Jacobi operator $R_X:\mathfrak{m}\to\mathfrak{m}$, $Y\mapsto R(Y,X)X$.

We first recall O'Neill's tensor $A:\mathcal{H}\times\mathcal{H}\to\mathcal{V}$, given by
\[
A_XY=(\nabla_XY)^v=\frac12[X,Y]^v.
\]
A straightforward computation yields
\[
A_XY=\frac12\left(\langle X,J_VY\rangle V+\langle X,J_WY\rangle W\right).
\]
Its adjoint tensor $A^\ast:\mathcal{H}\times\mathcal{V}\to\mathcal{H}$ is given by
\[
A^\ast_X(\lambda V+\mu W)
=-\frac12\left(\lambda J_VX+\mu J_WX\right).
\]

Applying O'Neill's curvature formulas (see \cite[p.~44]{GW}), we obtain, for every $X,Y\in\mathfrak{p}$,
\begin{equation}\label{curv.ad}
\begin{split}
\pi_*R_X(Y)
&=R^B(\pi_*Y,\pi_*X)\pi_*X
+\pi_*\left(2A^\ast_X(A_YX)-A^\ast_X(A_XY)\right)\\
&=-[[Y,X],X]
-\frac34\left(\langle Y,J_VX\rangle J_VX
+\langle Y,J_WX\rangle J_WX\right)\\
&=-\operatorname{ad}_X^2(Y)
-\frac34\left(\langle Y,J_VX\rangle J_VX
+\langle Y,J_WX\rangle J_WX\right).
\end{split}
\end{equation}

We shall also use the following expression for the curvature tensor in terms of the tensor $D=\nabla-\nabla^c$, where $\nabla^c$ denotes the canonical connection (see \cite[p.~11]{OR}). Since the homogeneous decomposition is naturally reductive, we have $U\equiv0$, and hence
\begin{equation}\label{curv.D}
\begin{split}
R(X,Y)Z
&=D_XD_YZ-D_YD_XZ-D_{[X,Y]_{\mathfrak{m}}}Z
-[[X,Y]_{\mathfrak{h}},Z]_{\mathfrak{m}}\\
&=\frac14[X,[Y,Z]_{\mathfrak{m}}]_{\mathfrak{m}}
-\frac14[Y,[X,Z]_{\mathfrak{m}}]_{\mathfrak{m}}
-\frac12[[X,Y]_{\mathfrak{m}},Z]_{\mathfrak{m}}
-[[X,Y]_{\mathfrak{h}},Z]_{\mathfrak{m}},
\end{split}
\end{equation}
for all $X,Y,Z\in\mathfrak{m}$.

Finally, let $X\in\mathfrak{p}$ and $Z\in\mathfrak{z}\subseteq\mathfrak{k}$. Applying \eqref{curv.D} together with the bracket relations of the Cartan decomposition, we obtain
\begin{equation}
\begin{split}
R_X(Z)
&=-\frac14[X,[Z,X]_{\mathfrak{m}}]_{\mathfrak{m}}
-\frac12[[Z,X]_{\mathfrak{m}},X]_{\mathfrak{m}}
-[[Z,X]_{\mathfrak{h}},X]_{\mathfrak{m}}\\
&=-\frac14[X,[Z,X]_{\mathfrak{p}}]_{\mathfrak{z}}
-\frac12[[Z,X]_{\mathfrak{p}},X]_{\mathfrak{z}}
=\frac14Z.
\end{split}
\end{equation}

\section{Reflective submanifolds}\label{reflective}

We say that a submanifold of a Riemannian manifold $M$ is reflective if it is a connected component of the fixed point set of an involutive isometry of $M$. In particular, every reflective submanifold is totally geodesic.

Reflective submanifolds play a central role in this work, since every maximal example appearing in our classification is of this type. We now describe the involutive isometries whose fixed point sets give rise to the reflective submanifolds occurring in Theorem~\ref{Teorema}.

As will be shown in Section~\ref{classif}, the complex structures preserved by a totally geodesic submanifold impose strong restrictions on its tangent space. For this reason, we record below the reflective submanifolds together with the complex structures they preserve.

\begin{itemize}

\item Let
\[
g=
\begin{pmatrix}
I_p & 0\\
0 & -I_q
\end{pmatrix},
\]
where $p+q=n+2$ and $p>2$. The fixed point set is
\[
M^g=\frac{SU(p)}{S(U(1)\times U(1)\times U(p-2))}
=F_{1,2}(\mathbb{C}^{p}).
\]
It is preserved by all four complex structures on $\mathfrak{p}$ introduced above. Moreover, it is $J$-holomorphic.

\item Let $g$ be the complex conjugation on $G$. The fixed point set is
\[
M^g=\frac{SO(n+2)}{S(O(1)\times O(1)\times O(n))}
=F_{1,2}(\mathbb{R}^{n+2}).
\]
It is preserved only by the complex structure $J_V$. Moreover, it is Lagrangian.

\item Let
\[
g=
\begin{pmatrix}
1 & 0 & 0 & 0\\
0 & -1 & 0 & 0\\
0 & 0 & I_p & 0\\
0 & 0 & 0 & -I_q
\end{pmatrix},
\]
where $p+q=n$. The fixed point set is
\[
M^g=
\frac{SU(p+1)\times SU(q+1)}
{S(U(1)\times U(p))\times S(U(1)\times U(q))}
=\mathbb{C}P^p\times\mathbb{C}P^q.
\]
It is preserved only by the complex structure $J_I$. Consequently, it is $J$-holomorphic.

\item Let $g:G\to G$ be given by
\[
g(X)=Y\overline{X}Y^{-1},
\]
where
\[
Y=
\begin{pmatrix}
J_0 & 0 & 0 & 0\\
0 & J_1 & 0 & 0\\
0 & 0 & \ddots & 0\\
0 & 0 & 0 & J_k
\end{pmatrix},
\]
with $k=\lfloor n/2\rfloor$ and
\[
J_i=
\begin{pmatrix}
0 & -1\\
1 & 0
\end{pmatrix}.
\]

The fixed point set of the induced isometry on $M$ is the twistor space of $\mathbb{H}P^k$,
\[
M^g=\frac{Sp(k+1)}{S^1\times Sp(k)}
=\mathbb{C}P^{2k+1}.
\]
It is preserved by the complex structures $J_V$, $J_W$, and $J_{VW}$.

\end{itemize}

\begin{rema}
The fixed point sets of the above involutions project, under the twistor fibration, onto the maximal totally geodesic submanifolds of the Grassmannian $G_2(\mathbb{C}^{n+2})$ presented in \cite{CN2}. The additional cases identified in \cite{K} do not arise from the involutions considered here. Thus, the reflective submanifolds considered here can be regarded as lifts of these totally geodesic submanifolds. Nevertheless, our classification is obtained independently of the classification of totally geodesic submanifolds of the Grassmannian.
\end{rema}

\subsection{Restricted root system and diagonalization}\label{diag}

Let $\mathfrak{a}=\operatorname{span}\{X_{13},X_{24}\}$
be a maximal abelian subspace of $\mathfrak{p}$. For $n>2$, the corresponding restricted root system is of type $BC_2$, and $\mathfrak{p}$ admits the orthogonal decomposition
\begin{equation}\label{dec}
\mathfrak{p}
=
\mathfrak{a}
\oplus
\mathfrak{p}_{\theta_1}
\oplus
\mathfrak{p}_{\theta_2}
\oplus
\mathfrak{p}_{\theta_1+\theta_2}
\oplus
\mathfrak{p}_{\theta_1-\theta_2}
\oplus
\mathfrak{p}_{2\theta_1}
\oplus
\mathfrak{p}_{2\theta_2},
\end{equation}
where $\theta_1$ and $\theta_2$ are the dual basis of $X_{13}$ and $X_{24}$, respectively, and
\begin{equation*}
\mathfrak{p}_{\alpha}
=
\left\{
Y\in\mathfrak{p}
:
-\operatorname{ad}_A^2(Y)
=
\alpha(A)^2Y,
\ \forall\,A\in\mathfrak{a}
\right\}.
\end{equation*}

\begin{rema}
For $n=2$, the restricted root system is of type $C_2$, and the root spaces $\mathfrak{p}_{\theta_1}$ and $\mathfrak{p}_{\theta_2}$ do not appear in the decomposition.
\end{rema}

Since the decomposition \eqref{dec} is orthogonal and
\begin{equation*}
J_V(\mathfrak{a})+J_W(\mathfrak{a})
=
\mathfrak{p}_{\theta_1+\theta_2}
\oplus
\mathfrak{p}_{\theta_1-\theta_2},
\end{equation*}
formula \eqref{curv.ad} immediately shows that, if $X\in\mathfrak{a}$ and $Y\in\mathfrak{p}_{\alpha}$ with $\alpha\in
\{\theta_1,\theta_2,2\theta_1,2\theta_2\},$
then $Y$ is an eigenvector of $R_X$ with eigenvalue $\alpha(X)^2$.

Therefore, it remains only to determine the eigenspaces of $R_X$ on
$\mathfrak{p}_{\theta_1+\theta_2}
\oplus
\mathfrak{p}_{\theta_1-\theta_2}.$
The following table describes the restricted root spaces appearing in the decomposition \eqref{dec}.

\begin{table}[h!] \centering \begin{tabular}{|c|c|c|} \hline \textbf{Space} & \textbf{Description} & \textbf{Dimension} \\ \hline $\mathfrak{a}$ & $span\{X_{13},X_{24}\}$ & $2$ \\ \hline $\mathfrak{p}_{\theta_1}$ & $span\{X_{15},...,X_{1(n+2)},Y_{15},...,Y_{1(n+2)}\}$ & $2n-4$ \\ \hline $\mathfrak{p}_{\theta_2}$ & $span\{X_{25},...,X_{2(n+2)},Y_{25},...,Y_{2(n+2)}\}$ & $2n-4$ \\ \hline $\mathfrak{p}_{\theta_1 +\theta_2}$ & $span\{X_{14}-X_{23},Y_{14}+Y_{23}\}$ & $2$ \\ \hline $\mathfrak{p}_{\theta_1 -\theta_2}$ & $span\{X_{14}+X_{23},Y_{14}-Y_{23}\}$ & $2$ \\ \hline $\mathfrak{p}_{2\theta_1}$ & $span\{Y_{13}\}$ & $1$ \\ \hline $\mathfrak{p}_{2\theta_2}$ & $span\{Y_{24}\}$ & $1$ \\ \hline \end{tabular} \caption{Restricted root decomposition} \label{tab:3} \end{table}

Write $X=\cos(t)X_{13}+\sin(t)X_{24}.$
The eigenvalues of $R_X$ on
$\mathfrak{p}_{\theta_1+\theta_2}\oplus\mathfrak{p}_{\theta_1-\theta_2}$
are
\begin{equation*}
\lambda_1^{\pm}
=
\frac{5\pm\sqrt{9+16\sin^2(2t)}}{8},
\qquad
\lambda_2^{\pm}
=
\frac{5\pm\sqrt{9+112\sin^2(2t)}}{8},
\end{equation*}
whose corresponding eigenvectors lie in
$\operatorname{span}\{X_{14},X_{23}\}$ and
$\operatorname{span}\{Y_{14},Y_{23}\}$, respectively.

We conclude this subsection by describing the eigenspace decomposition of $R_X$ in the two singular cases
\begin{equation*}
X=X_{13}
\qquad\text{and}\qquad
X=\frac{\sqrt2}{2}(X_{13}+X_{24}).
\end{equation*}


\begin{table}[h!] \centering \begin{tabular}{|c|c|c|} \hline \textbf{Eigenvalue} & \textbf{Eigenspace} & \textbf{Dimension} \\ \hline $0$ & $\mathfrak{a}+\mathfrak{p}_{\theta_2}+\mathfrak{p}_{2\theta_2}$ & $2n-1$ \\ \hline $\frac{1}{4}$ & $span\{V,W,J_VX,J_WX\}$ & $4$ \\ \hline $1$ & $\mathfrak{p}_{\theta_1}+span\{X_{14},Y_{14}\}$ & $2n-2$ \\ \hline $4$ & $\mathfrak{p}_{2\theta_1}$ & $1$ \\ \hline \end{tabular} \caption{Eigenspace decomposition of $R_X$ for $X=X_{13}$} \label{tab:1} \end{table}

For Table~\ref{tab:2}, we denote by
$X_-=X_{14}+X_{23},
Y_-=Y_{14}-Y_{23},$
the eigenvectors associated with the eigenvalues
$\lambda_1^{-}$ and $\lambda_2^{-}$, respectively.


\begin{table}[h!] \centering \begin{tabular}{|c|c|c|} \hline \textbf{Eigenvalue} & \textbf{Eigenspace} & \textbf{Dimension} \\ \hline $0$ & $\mathfrak{a}+span\{X_-\}$ & $3$ \\ \hline $\frac{1}{4}$ & $span\{V,W\}$ & $2$ \\ \hline $\frac{1}{2}$ & $\mathfrak{p}_{\theta_1}+\mathfrak{p}_{\theta_2}$ & $4n-8$ \\ \hline $2$ & $\mathfrak{p}_{2\theta_1}+\mathfrak{p}_{2\theta_2}+span\{J_WX\}$ & $3$ \\ \hline $\frac{5}{4}$ & $span\{J_VX\}$ & $1$ \\ \hline $-\frac{3}{4}$ & $span\{Y_-\}$ & $1$ \\ \hline \end{tabular} \caption{Eigenspace decomposition of $R_X$ for
$X=\frac{\sqrt{2}}{2}(X_{13}+X_{24})$} \label{tab:2} \end{table}

\subsection{Singular vectors and the isotropy representation}\label{singular}

In the previous subsection, we determined the eigenspace decomposition of $R_X$ only for the singular vectors corresponding to $t=0$ and $t=\pi/4$. In this subsection, we show that these are the only cases that need to be considered.

The isotropy representation $\varphi:K\longrightarrow GL(\mathfrak{p})$ identifies the orbit spaces $\mathfrak{p}/K=\mathfrak{a}/\mathcal{W}$, where $\mathcal{W}$ is the Weyl group and the maximal abelian subspace $\mathfrak{a}\subseteq\mathfrak{p}$ is a section of the representation $\varphi$.

The singular orbits are represented by the vectors $X=\cos(t)X_{13}+\sin(t)X_{24}$, with $t=0$ and $t=\pi/4$. Our next goal is to prove that, under suitable assumptions, the tangent space of a totally geodesic submanifold always contains a nonzero singular vector.

Let $\mathfrak{s}=T_o\Sigma$, where $\Sigma$ is a totally geodesic submanifold of $M$ with $\dim(\Sigma)\geq2$ and $\mathfrak{s}\cap\mathfrak{p}\neq{0}$. Assume that $X\in\mathfrak{s}\cap\mathfrak{a}$ is nonzero and that $\mathfrak{s}\cap\mathfrak{p}$ contains no nonzero singular vectors.

Lemma~\ref{l.tojo} is stated as in \cite{GKR} and follows directly from Tojo's criterion (\ref{Tojo}). The following two lemmas are adaptations of analogous results in \cite{GKR} to the present setting. For completeness, we include the proof of Lemma~\ref{autovetor}.

\begin{lemm}\label{l.tojo}(\cite{GKR})
Let $\mathfrak{m}=\bigoplus V_{\lambda}$ be the eigenspace decomposition of $(R_X)_o$. If $Y\in V_{\lambda}\cap\mathfrak{s}$ and $Z\in\mathfrak{s}$, then $\pi_{\mu}(-\nabla_YZ)\in\mathfrak{s}$ for every $\mu\neq\lambda$, where $\pi_{\mu}$ denotes the orthogonal projection onto the eigenspace $V_{\mu}$.
\end{lemm}

\begin{lemm}\label{inters}
If $\mathfrak{s}\cap\mathfrak{p}$ contains no nonzero singular vectors, then
$\mathfrak{s}\cap\mathfrak{a}=\mathbb{R}X,$
$\mathfrak{s}\cap\mathfrak{p}_{\theta_1}={0},$ $\mathfrak{s}\cap\mathfrak{p}_{\theta_2}={0},$ $\mathfrak{s}\cap\mathfrak{p}_{2\theta_1}={0},$ and 
$\mathfrak{s}\cap\mathfrak{p}_{2\theta_2}={0}.$
\end{lemm}

\begin{lemm}\label{autovetor}
Let $X_1^{\pm}$ and $X_2^{\pm}$ be unit eigenvectors associated with the eigenvalues $\lambda_1^{\pm}$ and $\lambda_2^{\pm}$, respectively. Then the following are equivalent:
\begin{enumerate}
\item[(i)] $X_1^+\in\mathfrak{s}$ or $X_1^-\in\mathfrak{s}$;
\item[(ii)] $V\in\mathfrak{s}$;
\item[(iii)] $X_1^+\in\mathfrak{s}$ and $X_1^-\in\mathfrak{s}$.
\end{enumerate}
The same statement holds with $(X_1^{\pm},V)$ replaced by $(X_2^{\pm},W)$.
\end{lemm}

\begin{proof}
Let $X=\cos(t)X_{13}+\sin(t)X_{24}\in\mathfrak{s}$, where $t\in(0,\pi/4)$.

\begin{itemize}
\item $(i)\Rightarrow(ii)$: Assume that $X_1^+\in\mathfrak{s}$; the proof for $X_1^-$ is analogous. Since $X_1^+\in V_{\lambda_1^+}$, Lemma~\ref{l.tojo} implies that
\[
\pi_{1/4}(-\nabla_{X_1^+}X)\in\mathfrak{s}.
\]
We claim that $\pi_{1/4}(-\nabla_{X_1^+}X)$ is a nonzero multiple of $V$.

Indeed, since $t\in(0,\pi/4)$, neither $J_VX$ nor $J_WX$ is an eigenvector of $R_X$. Hence,
\[
J_VX=aX_1^++bX_1^-,
\qquad
J_WX=cX_2^++dX_2^-,
\]
where $a,b,c,d$ are all nonzero. Therefore, for every $Y\in\mathfrak{p}$,
\[
\begin{aligned}
\langle-\nabla_{X_1^+}X,V\rangle
&=\frac12\langle J_VX,X_1^+\rangle
=\frac12a,\\
\langle-\nabla_{X_1^+}X,W\rangle
&=\frac12\langle J_WX,X_1^+\rangle
=0,\\
\langle-\nabla_{X_1^+}X,Y\rangle
&=-\frac12\langle[Y,X]_{\mathfrak{z}},X_1^+\rangle
=0.
\end{aligned}
\]
Hence, $V\in\mathfrak{s}$.

\item $(ii)\Rightarrow(iii)$: Assume that $V\in\mathfrak{s}$. Since $-\nabla_VX=\frac12J_VX$, Lemma~\ref{l.tojo} yields
\[
\pi_{\lambda_1^+}(-\nabla_VX)=aX_1^+\in\mathfrak{s},
\qquad
\pi_{\lambda_1^-}(-\nabla_VX)=bX_1^-\in\mathfrak{s}.
\]

\item $(iii)\Rightarrow(i)$: This implication is immediate.
\end{itemize}
\end{proof}

Since every nontrivial linear combination of $X_1^+$ and $X_1^-$ is
singular, Lemma~\ref{autovetor} implies that
$X_1^+,X_1^-,V\notin\mathfrak{s}$. The same argument shows that
$X_2^+,X_2^-,W\notin\mathfrak{s}$. Combining these observations with the
Lemma~\ref{inters} and the eigenspace decomposition of $R_X$, we obtain
$\mathfrak{s}=\mathbb{R}X,$
contradicting the assumption that $\dim(\mathfrak{s})\geq2$. Therefore,
$\mathfrak{s}$ contains a nonzero singular vector, as claimed.

\section{Classification}\label{classif}

By the results of the previous section, every totally geodesic submanifold satisfying
$\mathfrak{s}\cap\mathfrak{p}\neq\{0\}$ contains a nonzero singular vector.
Consequently, the classification reduces to the two singular $K$-orbits represented by
\begin{equation*}
X_{13}
\quad\text{and}\quad
\frac{\sqrt{2}}{2}(X_{13}+X_{24}).
\end{equation*}

The reflective submanifolds described in Section~\ref{reflective} are characterized by the existence of invariant complex structures. The next lemma shows that, under a natural assumption, the existence of a single complex plane inside $\mathfrak{s}$ already forces $\mathfrak{s}\cap\mathfrak{p}$ to be invariant under one of the complex structures $J_V$ or $J_W$.

\begin{lemm}\label{complexo}
Let $Y,J_VY\in\mathfrak{s}\cap\mathfrak{p}$, and assume that
$[Y,J_VY]_{\mathfrak{h}}$
centralizes
$span\{Y,J_VY\}^{\perp}\cap\mathfrak{s}\cap\mathfrak{p}$.
Then $\mathfrak{s}\cap\mathfrak{p}$ is a $J_V$-complex subspace of $\mathfrak{p}$.
The same conclusion holds after replacing $V$ by $W$.
\end{lemm}

\begin{proof}
The proof follows from the computation of
\[
R(Y,J_VY)Z,\qquad
Z\in
\operatorname{span}\{Y,J_VY\}^{\perp}
\cap\mathfrak{s}\cap\mathfrak{p},
\]
which shows that it is a nonzero multiple of $J_VZ$.
Since $\mathfrak{s}$ is $R$-invariant, it follows that
$J_VZ\in\mathfrak{s}\cap\mathfrak{p}$, as claimed.
\end{proof}

Throughout this subsection, we assume that $\mathfrak{s}$ is well-positioned with respect to the twistor fibration, namely,
\begin{equation*}
\mathfrak{s}
=
(\mathfrak{s}\cap\mathfrak z)
\oplus
(\mathfrak{s}\cap\mathfrak p).
\end{equation*}

\subsection{Classification for \texorpdfstring{$X=X_{13}$}{X=X13}}\label{caso1}

\subsubsection{The non-horizontal case}\label{flags}

We now consider the singular vector
$X=X_{13}.$
In this case, $V_{1/4}
=
span\{V,W,J_VX,J_WX\},$
as described in Table~\ref{tab:1}. Assume that $\mathfrak{s}\cap V_{1/4}\neq\{0\}$.

Let $\tau_t$ denote parallel transport along the geodesic
$\gamma(t)=\exp(tX)$
from $0$ to $t$.
Using Tojo's formula~\cite{T},
\begin{equation*}
d(L_{\exp(tX)})^{-1}\tau_t(Y)
=
e^{-t\nabla X^\ast}\cdot Y,
\end{equation*}
we obtain
\begin{equation*}
\begin{split}
\tau_t^{-1}(Y_p^\ast)
=
e^{t\nabla X^\ast}
\cdot
(d(L_g)^{-1}(Y_{go}^\ast))
=
e^{t\nabla X^\ast}
\cdot
(\operatorname{Ad}_{g^{-1}}Y)_o^\ast
=
e^{t\nabla X^\ast}
\cdot
(e^{-t\operatorname{ad}_X}Y)_{\mathfrak m},
\end{split}
\end{equation*}
where $p=g.o$ and $g=\exp(tX)$.

Consequently, parallel transport along $\gamma$ acts on $V_{1/4}$ as two simultaneous rotations, through the same angle, on the planes
$span\{V,J_VX\}$ and $span\{W,J_WX\}$.

Since $\mathfrak{s}$ is well-positioned, after parallel transport if necessary, we may assume that
\begin{equation*}
\mathfrak{s}\cap V_{1/4}\cap\mathfrak p\neq\{0\}.
\end{equation*}
Hence, there exists $\theta\in[0,2\pi]$ such that
\begin{equation*}
J_{U_\theta}X\in\mathfrak{s}, \mbox{ where }
U_\theta=\cos\theta\,V+\sin\theta\,W.
\end{equation*}

By Lemma~\ref{complexo}, $\mathfrak{s}\cap\mathfrak p$ is a $J_{U_\theta}$-complex subspace.

Therefore, the classification reduces to the possible values of
\begin{equation*}
\dim(\mathfrak{s}\cap V_{1/4}\cap\mathfrak p)
\quad\text{and}\quad
\dim(\mathfrak{s}\cap V_{1/4}\cap\mathfrak z),
\end{equation*}
subject to
\begin{equation*}
2
\ge
\dim(\mathfrak{s}\cap V_{1/4}\cap\mathfrak p)
\ge
\dim(\mathfrak{s}\cap V_{1/4}\cap\mathfrak z)
\ge
0.
\end{equation*}

\begin{enumerate}

\item Suppose that
\begin{equation*}
\dim(\mathfrak{s}\cap V_{1/4}\cap\mathfrak p)
=
\dim(\mathfrak{s}\cap V_{1/4}\cap\mathfrak z)
=
2.
\end{equation*}
Then $\mathfrak{s}$ is the tangent space of a lower-dimensional copy of the flag manifold,
\begin{equation*}
F_{1,2}(\mathbb{C}^{n+1}).
\end{equation*}

Indeed, since $\mathfrak{s}\cap\mathfrak p$ is both $J_V$- and $J_W$-complex, we may assume, up to the action of $H$, that
\begin{equation*}
\mathfrak{s}\cap\mathfrak p
=
span\{
X_{13},Y_{13},
X_{23},Y_{23},
\ldots,
X_{1k},Y_{1k},
X_{2k},Y_{2k}
\}.
\end{equation*}

Consequently,
\begin{equation*}
\Sigma
=
SU(n+1)\cdot o
=
\frac{SU(n+1)}
{S(U(1)\times U(1)\times U(n-1))}
=
F_{1,2}(\mathbb{C}^{n+1}).
\end{equation*}

\item Suppose that
\begin{equation*}
\dim(\mathfrak{s}\cap V_{1/4}\cap\mathfrak p)
=
\dim(\mathfrak{s}\cap V_{1/4}\cap\mathfrak z)
=
1.
\end{equation*}

Without loss of generality, assume that
\begin{equation*}
\mathfrak{s}\cap V_{1/4}
=
span\{J_VX,Z\}, \mbox{ where }
Z=\cos(t)V+\sin(t)W.
\end{equation*}

Then
\begin{equation*}
\nabla_ZR(J_VX,X,X)
=
\frac34\sin(t)\,J_VJ_WX.
\end{equation*}

Since $\mathfrak{s}$ is totally geodesic, the left-hand side belongs to $\mathfrak{s}$. Therefore, if $\sin(t)\neq0$, then
\begin{equation*}
J_WX\in\mathfrak{s},
\end{equation*}
which contradicts the assumption
\begin{equation*}
\dim(\mathfrak{s}\cap V_{1/4}\cap\mathfrak p)=1.
\end{equation*}

Hence, $\sin(t)=0,$
and $Z$ is parallel to $V$.

It follows that $\mathfrak{s}$ is the tangent space of the real flag manifold
\begin{equation*}
F_{1,2}(\mathbb{R}^{n+2}).
\end{equation*}

Indeed, since $\mathfrak{s}\cap\mathfrak p$ is $J_V$-complex, we may assume, up to the action of $H$, that
\begin{equation*}
\mathfrak{s}\cap\mathfrak p
=
span\{
X_{13},
X_{23},
\ldots,
X_{1k},
X_{2k}
\}.
\end{equation*}

Therefore,
\begin{equation*}
\Sigma
=
SO(n+2)\cdot o
=
\frac{SO(n+2)}
{S(O(1)\times O(1)\times O(n))}
=
F_{1,2}(\mathbb{R}^{n+2}).
\end{equation*}

\item Suppose that
\begin{equation*}
\dim(\mathfrak{s}\cap V_{1/4}\cap\mathfrak p)=2
\quad\text{and}\quad
\dim(\mathfrak{s}\cap V_{1/4}\cap\mathfrak z)=1.
\end{equation*}

This case cannot occur. Indeed, after applying a suitable element of the isotropy group $H$, we may assume that
$J_VX,\;J_WX,\;V\in\mathfrak{s}.$
Since
\begin{equation*}
R(J_VX,V)J_WX=-\frac{3}{4}W,
\end{equation*}
and $\mathfrak{s}$ is $R$-invariant, it follows that
$W\in\mathfrak{s},$
contradicting the assumption
\begin{equation*}
\dim(\mathfrak{s}\cap V_{1/4}\cap\mathfrak z)=1.
\end{equation*}

\item Suppose that
\begin{equation*}
\dim(\mathfrak{s}\cap V_{1/4}\cap\mathfrak p)=2
\quad\text{and}\quad
\dim(\mathfrak{s}\cap V_{1/4}\cap\mathfrak z)=0.
\end{equation*}

After applying a suitable element of the isotropy group $H$, we may assume that
\begin{equation*}
\mathfrak{s}
=
span\{
X_{13},Y_{13},
X_{23},Y_{23},
\ldots,
X_{1k},Y_{1k},
X_{2k},Y_{2k}
\}.
\end{equation*}

Assume first that $k>3$. Applying the parallel transport $\tau_t$, we obtain
\begin{equation*}
\mathfrak{s}
=
span\{
V,W,
X_{13},Y_{13},
X_{14},Y_{14},
X_{24},Y_{24},
\ldots,
X_{1k},Y_{1k},
X_{2k},Y_{2k}
\}.
\end{equation*}

Since
$X_{24}=J_VX_{14}\in\mathfrak{s},$
while
$X_{23}=J_VX_{13}\notin\mathfrak{s},$
the space $\mathfrak{s}\cap\mathfrak p$ is not $J_V$-invariant. This contradicts Lemma~\ref{complexo}, and therefore $\mathfrak{s}$ cannot be the tangent space of a totally geodesic submanifold.

If $k=3$, then $\mathfrak{s}\subseteq T_oF(\mathbb{C}^3),$
so $\Sigma$ is not maximal. Moreover, Proposition~7.5 of \cite{RN} shows that $F(\mathbb{C}^3)$ admits no totally geodesic submanifolds of codimension two, yielding another contradiction.

\item Finally, suppose that
\begin{equation*}
\dim(\mathfrak{s}\cap V_{1/4}\cap\mathfrak p)=1
\quad\text{and}\quad
\dim(\mathfrak{s}\cap V_{1/4}\cap\mathfrak z)=0.
\end{equation*}

The argument is analogous to the previous case. In particular, when $k=3$,
\begin{equation*}
\mathfrak{s}
=
span\{X_{13},X_{23}\}
\subseteq
T_oF(\mathbb{R}^3),
\end{equation*}
which implies that $\Sigma=\mathbb{R}P^2.$
Hence $\Sigma$ is not maximal (see \cite{RN}).

\end{enumerate}

\subsubsection{The horizontal case}\label{cp.cp}

Assume now that
\begin{equation*}
\mathfrak{s}\cap V_{1/4}=\{0\}.
\end{equation*}
Then $\mathfrak{s}$ is horizontal and we may write
\begin{equation*}
\mathfrak{s}
=
\mathbb{R}X
\oplus
\mathfrak{v}
\oplus
\mathfrak{w}
\oplus
\mathfrak{t}
\subseteq
\mathfrak{p},
\end{equation*}
where
\begin{equation*}
\mathfrak{w}\subseteq
V_1=
span\{X_{14},\ldots,X_{1(n+2)},Y_{14},\ldots,Y_{1(n+2)}\},
\end{equation*}
\begin{equation*}
\mathfrak{v}\subseteq
V_0=
span\{X_{24},\ldots,X_{2(n+2)},Y_{24},\ldots,Y_{2(n+2)}\},
\end{equation*}
and
\begin{equation*}
\mathfrak{t}\subseteq
V_4=
span\{Y_{13}\}.
\end{equation*}

\begin{lemm}\label{JI}
Each of the subspaces $\mathfrak w$ and $\mathfrak v$ is either totally real or
$J_I$-complex.
\end{lemm}

\begin{proof}
We prove the statement for $\mathfrak w$; the argument for $\mathfrak v$ is analogous.

Since $V_1$ is naturally identified with $\mathbb C^n$, the action of $U(n)$ on
$V_1$ is the standard one and is transitive on complex subspaces of fixed
dimension.

Suppose first that $\mathfrak w$ is totally real with respect to $J_I$, namely,
$J_I(\mathfrak w)\perp\mathfrak w.$
Then, after applying an element of $U(n)$, we may assume that
\begin{equation*}
\mathfrak w=
span\{X_{14},\ldots,X_{1(p+2)}\}.
\end{equation*}

Otherwise, after applying the $U(n)$-action, we may suppose that
\begin{equation*}
span\{X_{14},aY_{14}+bX_{15}\}
\subseteq
\mathfrak w,
\qquad
a\neq0.
\end{equation*}

Then,
\begin{equation*}
R(X_{14},aY_{14}+bX_{15})X_{14}
=
-4aY_{14}-bX_{15}
\in
\mathfrak s,
\end{equation*}
which implies that $Y_{14}\in\mathfrak s.$
Moreover, for every $k\ge5$,
\begin{equation*}
R(X_{14},Y_{14})X_{1k}
=
-[[X_{14},Y_{14}],X_{1k}]
=
-2Y_{1k}
\in
\mathfrak s.
\end{equation*}

Therefore,
\begin{equation*}
\mathfrak w=
span\{
X_{14},\ldots,X_{1(p+2)},
Y_{14},\ldots,Y_{1(p+2)}
\},
\end{equation*}
showing that $\mathfrak w$ is $J_I$-complex.
\end{proof}

The next lemma shows that the components $\mathfrak w$ and $\mathfrak v$ are orthogonal through the quaternionic structures.

\begin{lemm}
The subspaces $J_V(\mathfrak w)$ and $J_W(\mathfrak w)$ are orthogonal to $\mathfrak v$. Consequently,
\begin{equation*}
span\{X_{24},\ldots,X_{2(p+2)},Y_{24},\ldots,Y_{2(p+2)}\}
\perp
\mathfrak s.
\end{equation*}
\end{lemm}

\begin{proof}
Suppose that
$J_V(\mathfrak w)\not\perp\mathfrak v.$
Choose
\begin{equation*}
A=J_VB+C\in\mathfrak v,
\mbox{ where } 0\neq B\in\mathfrak w \mbox{ and }
C\perp J_V(\mathfrak w).
\end{equation*}
Then
\begin{equation*}
\begin{split}
\nabla_AB
=
-\frac12[A,B]_{\mathfrak m}
=
-\frac12
\left(
[J_VB,B]_{\mathfrak z}
+
[C,B]_{\mathfrak z}
\right)=
-\frac12
\left(
V
+
\langle[C,B]_{\mathfrak z},W\rangle W
\right).
\end{split}
\end{equation*}

Since
\begin{equation*}
A\in V_0,
\qquad
V+\langle[C,B]_{\mathfrak z},W\rangle W
\in V_{1/4},
\end{equation*}
Lemma~\ref{l.tojo} implies that
\begin{equation*}
V+\langle[C,B]_{\mathfrak z},W\rangle W
\in\mathfrak s,
\end{equation*}
contradicting the assumption that $\mathfrak s$ is horizontal.
The proof that
$J_W(\mathfrak w)\perp\mathfrak v$
is analogous.
\end{proof}

By Lemma~\ref{JI}, the subspace $\mathfrak v$ is also either totally real or
$J_I$-complex. Hence
\begin{equation*}
\mathfrak v=
span\{X_{2(p+3)},\ldots,X_{2(p+q+2)}\},
\end{equation*}
or
\begin{equation*}
\mathfrak v=
span\{
X_{2(p+3)},Y_{2(p+3)},
\ldots,
X_{2(p+q+2)},Y_{2(p+q+2)}
\},
\end{equation*}
where $0\le q\le n-p$.

Therefore, maximality forces
$\mathfrak t=
span\{Y_{13}\},$
and
\begin{equation*}
\begin{split}
\mathfrak s
=
span\{
X_{13},
\ldots,
X_{1(p+2)},
X_{2(p+3)},
\ldots,
X_{2(p+q+2)},
Y_{13},
\ldots,
Y_{1(p+2)},
Y_{2(p+3)},
\ldots,
Y_{2(p+q+2)}
\}.
\end{split}
\end{equation*}

Hence,
\begin{equation*}
\Sigma
=
(SU(p+1)\times SU(q+1)).o
=
\frac{SU(p+1)\times SU(q+1)}
{S(U(1)\times U(p))\times S(U(1)\times U(q))}
=
\mathbb CP^p\times\mathbb CP^q.
\end{equation*}
    
\subsection{Classification for \texorpdfstring{$X=\frac{\sqrt{2}}{2}(X_{13}+X_{24})$}{X=sqrt{2}{2}(X13+X24)}}\label{caso3}

From now on, let
\begin{equation*}
X=\frac{\sqrt2}{2}(X_{13}+X_{24}).
\end{equation*}

We may assume that $\mathfrak s$ contains no singular vectors of the same type as
$X_{13}$; otherwise, we are reduced to the cases treated in
Sections~\ref{flags} and~\ref{cp.cp}.

As in the previous subsection, our goal is to describe $\mathfrak s$ in terms of
the eigenspace decomposition of the Jacobi operator $R_X$.

The following three lemmas are adaptations of analogous results in \cite{GKR} to the present setting. We include the proofs for completeness.

\begin{lemm}\label{X-Y-}
Using the notation of Table~\ref{tab:2}, we have
$X_-,Y_-\notin\mathfrak s.$
\end{lemm}

\begin{proof}
Observe that
$[X,X_-]=[X,Y_-]=0.$
Hence,
$span\{X,X_-\}
\mbox{ and }
span\{X,Y_-\}$
are conjugate to $\mathfrak a$ and therefore contain singular vectors of the same
type as $X_{13}$, contradicting our assumption.
\end{proof}

\begin{lemm}\label{vertical.complexo}
We have
$V\in\mathfrak s
\iff
J_VX=X_1^+\in\mathfrak s,$
and
$W\in\mathfrak s
\iff
J_WX=X_2^+\in\mathfrak s.$
\end{lemm}

\begin{proof}
The proof is similar to that of Lemma~\ref{autovetor}.
\end{proof}

\begin{lemm}\label{lema.V1/2}
If $0\neq Y\in\mathfrak s\cap V_{1/2},$
then, up to the action of $H$, $Y=X_{15}+X_{26}.$
\end{lemm}

\begin{proof}
Using the action of $SU(n-2)\subseteq H$, we may assume that
\begin{equation*}
Y=X_{15}+aX_{25}+bY_{25}+cX_{26},
\qquad c\geq0.
\end{equation*}
Then
$\sqrt{2}\,R_Y(X)
=
X_{13}+aX_--bY_-+(a^2+b^2+c^2)X_{24}$
belongs to
$\mathfrak{s}\cap
(\mathfrak{a}+\mathbb{R}X_-+\mathbb{R}Y_-)$.
By Lemma~\ref{X-Y-}, we obtain $a=b=0$ and $c=1$, and the result follows.
\end{proof}

It remains to determine the possible intersections with the eigenspace
\begin{equation*}
V_2=
span\{J_WX,J_V(J_WX),J_IX\}.
\end{equation*}

We distinguish three cases according to the dimension of
$\mathfrak{s}\cap\mathfrak{z}$.

\subsubsection{Case 1: a twistor space}

Assume first that $\mathfrak{s}\cap\mathfrak{z}=\mathfrak{z}.$

By Lemma~\ref{vertical.complexo},
\begin{equation*}
J_VX,\;J_WX\in\mathfrak{s},
\end{equation*}
and Lemma~\ref{complexo} implies that
\begin{equation*}
J_V(J_WX)\in\mathfrak{s}.
\end{equation*}

On the other hand,
\begin{equation*}
J_IX\notin\mathfrak{s},
\end{equation*}
since
\begin{equation*}
R(X,J_WX)J_IX
=
-\frac32J_W(J_IX)
=
-\frac32X_-,
\end{equation*}
and $X_-\notin\mathfrak{s}$ by Lemma~\ref{X-Y-}.

Hence
\begin{equation*}
\begin{split}
\mathfrak{s}
=
\operatorname{span}\{
&
V,W,X,J_VX,J_WX,J_V(J_WX),
X_{15}+X_{26},
X_{25}-X_{16},
Y_{25}+Y_{16},
Y_{15}-Y_{26},
\dots,\\
&
X_{1(2k+1)}+X_{2(2k+2)},
X_{2(2k+1)}-X_{1(2k+2)},
Y_{2(2k+1)}+Y_{1(2k+2)},
Y_{1(2k+1)}-Y_{2(2k+2)}
\}.
\end{split}
\end{equation*}

Therefore,
\begin{equation*}
\Sigma
=
Sp(k+1)\cdot o
=
\frac{Sp(k+1)}{S^1\times Sp(k)}
=
\mathbb{C}P^{2k+1},
\end{equation*}
which is the twistor space of $\mathbb{H}P^k$ (see~\cite{JB}).

\subsubsection{Case 2: Horizontal case}

Assume now that $\mathfrak{s}\cap\mathfrak{z}=\{0\}.$

By Lemma~\ref{vertical.complexo},
\begin{equation*}
J_VX,\;J_WX\notin\mathfrak{s}.
\end{equation*}

Let
\begin{equation*}
Y=aY_{13}+bY_{24}+cJ_WX
\in
\mathfrak{s}\cap V_2.
\end{equation*}

Then
\begin{equation*}
\nabla_XR(Y,X,X)=\frac{c}{2}W,
\end{equation*}
which implies that $c=0.$
Moreover,
\begin{equation*}
R(X,aY_{13}+bY_{24})(aY_{13}+bY_{24})
=
2\sqrt2\,(a^2X_{13}+b^2X_{24})
\in
\mathfrak{s}\cap\mathfrak{a}
=
\mathbb{R}X.
\end{equation*}

Hence, $a^2=b^2,$
and therefore exactly one of the following two possibilities occurs.

\medskip

\noindent
\textbf{Case (i).}
\begin{equation*}
\mathfrak{s}\cap V_2
=
\operatorname{span}\{J_V(J_WX)\}.
\end{equation*}

If
\begin{equation*}
Y=X_{1(2k-1)}+X_{2(2k)}
\in
\mathfrak{s},
\end{equation*}
then
\begin{equation*}
R(X,J_V(J_WX))Y
=
-J_V(J_WY)
\in
\mathfrak{s}.
\end{equation*}

Consequently, $\mathfrak{s}$ is contained in the tangent space obtained in
Case~1, and therefore it is not maximal.

\medskip

\noindent
\textbf{Case (ii).}
\begin{equation*}
\mathfrak{s}\cap V_2
=
\operatorname{span}\{J_IX\}.
\end{equation*}

In this case, $\Sigma=\exp(\mathfrak{s})$ is the diagonal of
$\mathbb{C}P^k\times\mathbb{C}P^k$
(cf.~Section~\ref{cp.cp}) and hence is not maximal.

Indeed, if
\begin{equation*}
Y=X_{1(2k-1)}+X_{2(2k)}
\in
\mathfrak{s},
\end{equation*}
then
\begin{equation*}
R(X,J_IX)Y
=
-J_IY
\in
\mathfrak{s}.
\end{equation*}

Therefore,
\begin{equation*}
\mathfrak{s}
=
\operatorname{span}
\{
X_{13}+X_{24},
Y_{13}+Y_{24},
\dots,
X_{1(2k+1)}+X_{2(2k+2)},
Y_{1(2k+1)}+Y_{2(2k+2)}
\}.
\end{equation*}

\subsubsection{Case 3: One-dimensional vertical component}

Finally, assume that
\begin{equation*}
\mathfrak{s}\cap\mathfrak{z}
=
\operatorname{span}\{U_\theta\},
\mbox{ where }
U_\theta
=
\cos\theta\,V+\sin\theta\,W.
\end{equation*}

An adaptation of Lemma~\ref{vertical.complexo} shows that
\begin{equation*}
U_\theta\in\mathfrak{s}
\quad\Longrightarrow\quad
J_{U_\theta}X\in\mathfrak{s}.
\end{equation*}

Since $\mathfrak{s}$ is generated by eigenvectors of $R_X$, it follows that
\begin{equation*}
\theta=\frac{k\pi}{2},
\qquad
k\in\mathbb{Z}.
\end{equation*}

Hence, $U_\theta=V \mbox{ or }
U_\theta=W.$
Arguing exactly as in the previous cases, we obtain
\begin{equation*}
\mathfrak{s}
=
\operatorname{span}
\{
U_\theta,
X,
J_{U_\theta}X,
\dots,
X_{1(2k+1)}+X_{2(2k+2)},
J_{U_\theta}(X_{1(2k+1)}+X_{2(2k+2)})
\}.
\end{equation*}

Therefore, $\mathfrak{s}$ is contained in the tangent space obtained in
Case~1, and consequently it is not maximal.

\subsection{Non-well-positioned submanifolds}\label{nbp}

We now assume that $\mathfrak s$ is not well-positioned with respect to the twistor fibration. Our goal is to show that this situation is considerably more rigid than the well-positioned case and produces no new maximal examples.

By the results of Section~\ref{singular}, every totally geodesic submanifold contains a nonzero singular vector whenever $\mathfrak s\cap\mathfrak p\neq\{0\}$. Hence, exactly one of the following situations occurs:
\begin{equation*}
    (i)\ \mathfrak s\cap\mathfrak p=\{0\} \qquad
    (ii)\ X_{13}\in\mathfrak s \qquad
    (iii)\ \frac{\sqrt2}{2}(X_{13}+X_{24})\in\mathfrak s.
\end{equation*}

The third possibility cannot occur in the present setting, since Table~\ref{tab:2} shows that every subspace containing a singular vector of this type is automatically well-positioned. Therefore, it remains to analyze only cases~(i) and~(ii).

\subsubsection{Case (i): $\mathfrak s\cap\mathfrak p=\{0\}$}

Since $\mathfrak s$ has trivial intersection with the horizontal distribution, we necessarily have $\dim\mathfrak s=2.$
Choose an orthogonal basis
$\mathfrak s=\operatorname{span}\{Z_1,Z_2\}.$
Using the isotropy action, we may assume that
\begin{equation*}
    Z_1=V+aX_{13}+bX_{23}+cY_{23}+dX_{24},
    \qquad a\neq0.
\end{equation*}
Since $Z_2$ is diagonal and an eigenvector of $R_{Z_1}$, it follows that
\begin{equation*}
\begin{split}
    Z_2=eV+fW+gX_{13}+hY_{13}+iX_{23}+jY_{23}+kX_{14}+lY_{14}+mX_{24}+nY_{24},
\end{split}
\end{equation*}
where
$e^2+f^2>0.$

In particular, $\Sigma$ is contained in the totally geodesic submanifold
$F_{1,2}(\mathbb C^4)$. Consequently, $\Sigma$ can be maximal only when
$M=F_{1,2}(\mathbb C^4)$.

\begin{rema}
If $d=0$, then the conditions that $Z_2$ is diagonal and an eigenvector of
$R_{Z_1}$ force
\begin{equation*}
    k=l=m=n=0.
\end{equation*}
Hence, $\Sigma\subseteq F(\mathbb C^3),$
and therefore $\Sigma$ is congruent to one of the totally geodesic
submanifolds classified in \cite{RN}. In particular, it is not maximal.
\end{rema}

By the previous remark, we may assume that $d\neq0.$

To analyze this remaining situation, we apply the characterization of totally geodesic surfaces obtained in \cite{RN}. This criterion imposes strong algebraic restrictions on $Z_1$ and $Z_2$, which eventually determine them uniquely.

\begin{prop}\label{prop.surface}\cite{RN}
Let $X,Y\in\mathfrak m$ be orthogonal vectors such that $\mathfrak s=\operatorname{span}\{X,Y\}$
is totally geodesic. Then
\begin{equation*}
    \operatorname{span}\{D_X^kY:k\ge0\}
    \subseteq
    \ker(R_X-\kappa\|X\|^2\mathrm{Id})
    \cap
    \bigcap_{j=0}^{\infty}\ker(C_X^j),
\end{equation*}
where $D=\nabla-\nabla^c=\frac12[X,Y]_{\mathfrak m},$
$C_X^j(Y)=\nabla R(X,\ldots,X,Y,X)$, and $\kappa$ denotes the sectional curvature of $\mathfrak s$.
\end{prop}

Using this criterion, we compute
\begin{equation*}
\begin{split}
[Z_1,[Z_1,Z_2]_{\mathfrak m}]_{\mathfrak m}
={}&-Z_2+eZ_1-e\|Z_1^h\|^2V+\left(-ah+bj-ci+dn-f(\|Z_1^h\|^2-1)\right)W\\
&-fJ_VJ_WZ_1^h
+(bg+ch-ai+dk)J_VZ_1^h+(bh-cg+aj+dl)J_WZ_1^h\\
={}&-Z_2+eZ_1+\alpha V+\beta W+\gamma J_VJ_WZ_1^h
+\delta J_VZ_1^h+\epsilon J_WZ_1^h.
\end{split}
\end{equation*}

Hence,
\begin{equation*}
C_{Z_1}
\left(
\alpha V+\beta W+\gamma J_VJ_WZ_1^h
+\delta J_VZ_1^h+\epsilon J_WZ_1^h
\right)=0.
\end{equation*}

Denote by $(X_{ij})$ the equation
\begin{equation*}
\left\langle
C_{Z_1}
(\alpha V+\beta W+\gamma J_VJ_WZ_1^h+\delta J_VZ_1^h+\epsilon J_WZ_1^h),
X_{ij}
\right\rangle=0.
\end{equation*}

This yields the system
\begin{equation*}
\begin{cases}
a\alpha-b\delta=0 &(X_{14})\\
b\alpha-c\beta+a\delta=0 &(X_{24})\\
(-3+a^2)\beta-3(b^2+c^2+d^2)\gamma-2ac\delta-3ab\epsilon=0 &(Y_{14})\\
2ac\alpha+2ab\beta+3ab\gamma+a^2\delta-3(b^2+c^2+d^2)\epsilon=0 &(Y_{24})\\
3c\beta+3c(b^2+c^2+d^2)\gamma+a(3c^2+d^2)\delta+3abc\epsilon=0 &(X_{13}).
\end{cases}
\end{equation*}

We now solve this system step by step.

Using equations $(X_{14})$ and $(X_{24})$, we obtain
\begin{equation*}
    \alpha=\frac{bc}{a^2+b^2}\beta,
    \qquad
    \delta=\frac{ac}{a^2+b^2}\beta.
\end{equation*}

Substituting these expressions into $(X_{13})$ gives
\begin{equation*}
c\left(
\left(
3+\frac{a^2(3c^2+d^2)}{a^2+b^2}
\right)\beta
+
3(b^2+c^2+d^2)\gamma
+
3ab\epsilon
\right)=0.
\end{equation*}

Combining this equation with $(Y_{14})$ yields $c=0$ or $\beta=0$.
If $\beta=0$, then $\alpha=\gamma=\delta=\epsilon=0$, which implies $e=f=0$, contradicting $e^2+f^2>0$.
Therefore, $c=0$.
Consequently,
\[
g=i=k=m=0,
\qquad
\beta\neq0,
\qquad
\alpha=\delta=0,
\]
which implies $e=0$.
After rescaling, we may assume $f=1$.
Moreover,
\[
\gamma=
\frac{-3(b^2+d^2)-a^2b^2+a^2d^2}
{3((b^2+d^2)^2+a^2b^2)}
\beta,
\mbox{ and }
\epsilon=
\frac{ab(2b^2+2d^2-3+a^2)}
{3((b^2+d^2)^2+a^2b^2)}
\beta.
\]
Therefore,
\[
\ker(C_{Z_1})
\cap
\operatorname{span}
\{V,W,J_VJ_WZ_1^h,J_VZ_1^h,J_WZ_1^h\}
\]
is one-dimensional.

Next, observe that
\[
\begin{split}
[Z_1,[Z_1,[Z_1,Z_2]_{\mathfrak m}]_{\mathfrak m}]_{\mathfrak m}
={}&-[Z_1,Z_2]_{\mathfrak m}
-|Z_1^h|^2\epsilon W
-\epsilon J_VJ_WZ_1^h
+(\beta+\gamma)J_WZ_1^h\\
={}&-[Z_1,Z_2]_{\mathfrak m}
+\beta'W+\gamma'J_VJ_WZ_1^h+\epsilon'J_WZ_1^h,
\end{split}
\]
which also belongs to $\ker(C_{Z_1})$. Hence,
\[
\begin{cases}
\beta'=\lambda\beta,\\
\gamma'=\lambda\gamma,\\
\epsilon'=\lambda\epsilon,
\end{cases}
\]
from which we conclude that
$\lambda=0$, $\epsilon=0$, and $\beta=-\gamma$.

Substituting these relations into equation $(Y_{24})$, we obtain
\[
2ab\beta+3ab\gamma=0,
\]
which implies $b=0$.
Therefore, $j=l=0$, and
\[
\gamma=
\frac{-3+a^2}{3d^2}\beta
=-\beta,
\]
so that $-3+a^2+3d^2=0.$

Furthermore, since both
$Z_2
 \mbox{ and }
[Z_1,[Z_1,Z_2]_{\mathfrak m}]_{\mathfrak m}
=-Z_2+\beta W-\beta J_VJ_WZ_1^h$ belong to the same eigenspace of $R_{Z_1}$, Proposition~\ref{prop.surface} implies that
$W-J_VJ_WZ_1^h$ is also an eigenvector of $R_{Z_1}$.
Since
\[
\begin{split}
R_{Z_1}(W-J_VJ_WZ_1^h)
={}
2(2-a^2-d^2)W-2(-1+2a^2)aY_{13}
+2(-1+2d^2)dY_{24},
\end{split}
\]
we conclude that
$a^2=d^2=\frac34$.

Finally,
\[
Z_1=V+aX_{13}+dX_{24},
\qquad
Z_2=W+hY_{13}+nY_{24}.
\]
The conditions that $Z_2$ is an eigenvector of $R_{Z_1}$ and belongs to
$\ker(C_{Z_1})$ imply
$h=-a$ and $n=d$.

If $a=d=\frac{\sqrt3}{2}$, then
$\Sigma\subseteq\mathbb CP^3$
(see Section~\ref{caso3}) and therefore $\Sigma$ is not maximal.

On the other hand, if $d=-a=-\frac{\sqrt3}{2}$, then, after applying the isotropy action,
\[
Z_1=V+aX_{14}-aX_{23},
\qquad
Z_2=W-aY_{14}-aY_{23},
\]
and again
$\Sigma\subseteq\mathbb CP^3$.

In either case, $\Sigma$ is contained in a totally geodesic copy of $\mathbb CP^3$ and therefore cannot be maximal.

\subsubsection{Case (ii): $X_{13}\in\mathfrak s$}

Assume that $X=X_{13}\in\mathfrak s.$
In this situation, the intersection of $\mathfrak s$ with the eigenspace
$V_{1/4}$ determines whether $\mathfrak s$ is well-positioned.

\begin{prop}
If $X=X_{13}\in\mathfrak s$ and $\mathfrak s$ is not well-positioned, then
$\dim(\mathfrak s\cap\mathfrak p\cap V_{1/4})=0.$
\end{prop}

\begin{proof}
If
$\dim(\mathfrak s\cap\mathfrak p\cap V_{1/4})=2,$
then $\mathfrak s$ is automatically well-positioned, since it is generated by
eigenvectors of $R_X$ and
$V_{1/4}
=
\operatorname{span}
\{V,W,J_VX,J_WX\}.$

Suppose instead that
$\dim(\mathfrak s\cap\mathfrak p\cap V_{1/4})=1,$
and, up to the isotropy action, assume that
$J_VX\in\mathfrak s.$
Let
\begin{equation*}
Z=aV+bW+cJ_WX
\in
\mathfrak s\cap V_{1/4}.
\end{equation*}

Then
\begin{equation*}
R(X,J_VX)Z
=
-\frac34c\,J_VJ_WX.
\end{equation*}

Since $\mathfrak s$ is $R$-invariant, we must have $c=0$, so that
$Z\in\operatorname{span}\{V,W\}.$
Hence $\mathfrak s$ is well-positioned, contradicting the assumption.
\end{proof}

The same argument also shows that, if $\mathfrak s$ is not well-positioned,
then
\begin{equation*}
\dim(\mathfrak s\cap\mathfrak p\cap V_{1/4})
=
\dim(\mathfrak s\cap\mathfrak z\cap V_{1/4})
=
0.
\end{equation*}

Therefore, $\mathfrak s\cap V_{1/4}$
is generated entirely by diagonal vectors.

We now distinguish according to the dimension of
$\mathfrak s\cap V_{1/4}$.

\paragraph{The two-dimensional case.}

Assume that
$\mathfrak s\cap V_{1/4}
=
\operatorname{span}\{Z_1,Z_2\},$ where
\begin{equation*}
Z_1
=
V+\alpha J_VX+\delta J_WX,
\qquad
Z_2
=
W+\beta J_VX+\gamma J_WX.
\end{equation*}

Using the curvature formula (\ref{curv.D}), we obtain

\begin{equation*}
R(X,Z_1)Z_2
=
\frac34
(1-\alpha\gamma+\beta\delta)
J_VJ_WX
+
\frac14
(\alpha\beta+\gamma\delta)X.
\end{equation*}

Suppose first that
\begin{equation*}
1-\alpha\gamma+\beta\delta\neq0.
\end{equation*}

Then $J_VJ_WX\in\mathfrak s.$
Moreover,
\begin{equation*}
R(X,Z_1)J_VJ_WX
=
-\frac34W
+\frac34\alpha J_WX
-\frac34\delta J_VX,
\end{equation*}

which implies
\begin{equation}\label{eq.jvjw}
\gamma=-\alpha,
\qquad
\beta=\delta.
\end{equation}

On the other hand,
\begin{equation*}
\nabla_{Z_1}R(J_VJ_WX,X,X)
=
\frac34(-J_WX-\alpha W+\delta V)
\in\mathfrak s.
\end{equation*}

Hence it must be a linear combination of $Z_1$ and $Z_2$, namely,
\begin{equation*}
\nabla_{Z_1}R(J_VJ_WX,X,X)
=
\delta Z_1-\alpha Z_2.
\end{equation*}

Using (\ref{eq.jvjw}), we obtain
\begin{equation*}
\begin{cases}
0=\alpha(\beta-\delta),\\
\dfrac34
=
-(\delta^2+\alpha^2),
\end{cases}
\end{equation*}

which is impossible.

Therefore,
\begin{equation}\label{eq.estrela}
1-\alpha\gamma+\beta\delta=0.
\end{equation}

Furthermore,
\begin{equation*}
\nabla_{Z_2}R(Z_1,X,X)
=
\frac34(\alpha+\gamma)J_VJ_WX,
\end{equation*}

and hence $\gamma=-\alpha.$

Equation (\ref{eq.estrela}) becomes $1+\alpha^2+\beta\delta=0,$ or equivalently, $\beta
=
-\frac{1+\alpha^2}{\delta},
\delta\neq0.$

Thus, $Z_1=V+\alpha J_VX+\delta J_WX,$ and $Z_2=W-\frac{1+\alpha^2}{\delta}J_VX-\alpha J_WX.$
Since $\mathfrak s$ is $R$-invariant,
$R_{Z_1}(Z_2)\in\operatorname{span}\{Z_1,Z_2\}.$
A direct computation gives
\begin{equation*}
\begin{split}
R_{Z_1}(Z_2)=&
\frac{1+\alpha^4+25\delta^2+\alpha^2(2+\delta^2)}
{4\delta^2}V
+
\frac{\alpha^3\delta+\alpha(\delta+\delta^3)}
{4\delta^2}W\\
&
-6\alpha J_VX
+
\frac{25+25\alpha^2+\delta^2}
{4\delta}J_WX.
\end{split}
\end{equation*}

Comparing with the decomposition in the basis $\{Z_1,Z_2\}$ yields $\alpha=0,
\delta^2=1.$

Therefore,
\begin{equation*}
Z_1=V+\delta J_WX,
\qquad
Z_2=W-\delta J_VX,
\qquad
\delta=\pm1.
\end{equation*}

At this stage, $X,Z_1,Z_2\in\mathfrak s,$ whereas
$J_VX,
J_WX,
J_VJ_WX
\notin\mathfrak s.$

Assume now that
$\dim\mathfrak s>3.$
Since $\mathfrak s$ is generated by eigenvectors of $R_X$, Table~\ref{tab:2}
implies the existence of a horizontal vector
\begin{equation*}
Y
=
\sum_{j=4}^{n+2}
a_{ij}X_{ij}
+
b_{ij}Y_{ij},
\qquad
i=1
\text{ or }
2,
\end{equation*}

belonging to $\mathfrak s$.
However,
\begin{equation}\label{eq.Y}
\nabla_YR(Z_1,X,X)
=
\begin{cases}
-\dfrac18J_VY,
&
i=1,
\\[1ex]
\dfrac38J_VY,
&
i=2.
\end{cases}
\end{equation}

Since
$[Y,J_VY]_{\mathfrak h}=0,$
Lemma~\ref{complexo} implies that
$\mathfrak s$ is $J_V$-complex, forcing $J_VX\in\mathfrak s,$
a contradiction.

Consequently, $\mathfrak s=\operatorname{span}\{X,Z_1,Z_2\}.$
In particular, $\Sigma\subseteq F(\mathbb C^3),$
and therefore $\Sigma$ is not maximal.

\paragraph{The one-dimensional case.}

Assume now that $\mathfrak s\cap V_{1/4}=\operatorname{span}\{Z\},$
where, up to the isotropy action,
\begin{equation*}
Z
=
V+aJ_VX+bJ_WX.
\end{equation*}

We first show that $J_VJ_WX\notin\mathfrak s.$
Indeed, otherwise,
\begin{equation*}
\nabla_ZR(J_VJ_WX,X,X)
=
\frac34(bV-aW-J_WX)
\in\mathfrak s.
\end{equation*}
Since this vector must be proportional to $Z$, we obtain a contradiction.
Hence $J_VJ_WX\notin\mathfrak s.$

Finally, applying the same argument used in (\ref{eq.Y}), we conclude that
\begin{equation*}
\mathfrak s
=
\operatorname{span}\{X,Z\}.
\end{equation*}

Therefore, $\Sigma\subseteq F(\mathbb C^3),$
and once again $\Sigma$ is not maximal.
This completes the proof that, for $n\ge2$, every maximal totally geodesic submanifold of
\begin{equation*}
M
=
SU(n+2)
/S(U(1)\times U(1)\times U(n))
\end{equation*}
is necessarily well-positioned with respect to the twistor fibration.

\section{Further classification results}\label{further}

In this section we prove the two classification results stated in the Introduction.

\begin{proof}[Proof of Theorem~\ref{tg.cp2k+1}]
Let $\Sigma$ be a maximal connected complete totally geodesic submanifold of
$\mathbb{C}P^{2k+1}$ passing through the origin, and let
\[
\mathfrak{s}=T_o\Sigma.
\]
Since $\mathbb{C}P^{2k+1}$ arises as the twistor space described in case~(3) of
Theorem~\ref{Teorema}, the possible tangent spaces are precisely those obtained
in our classification.

Assume first that
$\mathfrak{s}\cap\mathfrak p\neq\{0\}.$
By Theorem~\ref{Teorema}, we may assume that
\[
X_1=\frac{\sqrt2}{2}(X_{13}+X_{24})\in\mathfrak{s},
\]
and define
\[
X_j=\frac{\sqrt2}{2}(X_{1(2j+1)}+X_{2(2j+2)}).
\]

According to the possible vertical components of $\mathfrak{s}$, three cases may occur.

\begin{itemize}

\item If $\mathfrak{s}\cap\mathfrak{z}=\mathfrak{z},$ then
\begin{equation*}
\begin{split}
\mathfrak{s}
=
span\{
V,W,
X_1,J_VX_1,J_WX_1,J_VJ_WX_1,\ldots,
X_{k-1},J_VX_{k-1},J_WX_{k-1},J_VJ_WX_{k-1}
\},
\end{split}
\end{equation*}
and hence
\[
\Sigma
=
Sp(k).o
=
\frac{Sp(k)}{S^1\times Sp(k-1)}
=
\mathbb{C}P^{2k-1}.
\]

\item If $\mathfrak{s}\cap\mathfrak{z}=span\{V\},$ then
\[
\mathfrak{s}
=
span\{V,X_1,J_VX_1,\ldots,X_k,J_VX_k\},
\]
and therefore
\[
\Sigma
=
U(k+1).o
=
\frac{U(k+1)}{\mathbb{Z}_2\times U(k)}
=
\frac{\mathbb{S}^{2k+1}_{\mathbb{C},1/2}(\sqrt2)}{\mathbb{Z}_2}
=
\mathbb{R}P^{2k+1}_{\mathbb{C},1/2}(\sqrt2).
\]
The corresponding involutive isometry is induced by complex conjugation with
respect to the quaternionic units $i$ and $k$.

\item If $\mathfrak{s}\cap\mathfrak{z}=\{0\},$ then
\[
\mathfrak{s}
=
span\{X_1,J_VJ_WX_1,\ldots,X_k,J_VJ_WX_k\},
\]
and consequently
\[
\Sigma
=
U(k+1).o
=
\frac{U(k+1)}{U(1)\times U(k)}
=
\mathbb{C}P^k.
\]
The corresponding involutive isometry is induced by conjugation with respect to
the quaternionic unit $j$.

\end{itemize}

Finally, if
$\mathfrak{s}\cap\mathfrak p=\{0\},$
then $\Sigma$ cannot be maximal. Indeed, in this case $\Sigma$ is either a
fiber of the twistor fibration or one of the non-well-positioned
submanifolds described in Section~\ref{nbp}, both of which are contained in
$\mathbb{C}P^3$.
\end{proof}

\begin{proof}[Proof of Theorem~\ref{real.flag}]
By the classification of maximal totally geodesic submanifolds of the real
Stiefel manifold obtained in \cite{GKR}, there are three possible tangent
spaces for a maximal totally geodesic submanifold.

Since the canonical double covering
\[
V_2(\mathbb{R}^{n+2})
\longrightarrow
F_{1,2}(\mathbb{R}^{n+2})
\]
is a local isometry, these tangent spaces also determine the maximal totally
geodesic submanifolds of the real partial flag manifold.

\begin{itemize}

\item If
$\mathfrak{s}
=
span\{V,X_{13},X_{23},\ldots,X_{1(n+1)},X_{2(n+1)}\},$ then
\[
\Sigma
=
SO(n+1).o
=
\frac{SO(n+1)}
{S(O(1)\times O(1)\times O(n-1))}
=
F_{1,2}(\mathbb{R}^{n+1}).
\]

\item If
$\mathfrak{s}
=
span\{X_{13},\ldots,X_{1(p+2)},
X_{2(p+3)},\ldots,X_{2(p+q+2)}\},$ then
\[
\Sigma
=
(SO(p+1)\times SO(q+1)).o
=
\frac{SO(p+1)\times SO(q+1)}
{S(O(1)\times O(p))\times S(O(1)\times O(q))}
=
\mathbb{R}P^p\times\mathbb{R}P^q.
\]

\item If
$\mathfrak{s}
=
span\{
V,
X_{13}+X_{24},
X_{23}-X_{14},
\ldots,
X_{1(2k+1)}+X_{2(2k+2)},
X_{2(2k+1)}-X_{1(2k+2)}
\},$
then
\[
\Sigma
=
U(k+1).o
=
\frac{\mathbb{S}^{2k+1}_{\mathbb{C},1/2}(\sqrt2)}
{\mathbb{Z}_2}
=
\mathbb{R}P^{2k+1}_{\mathbb{C},1/2}(\sqrt2).
\]

\end{itemize}
\end{proof}

\bibliographystyle{amsplain} 
\bibliography{referencias} 

\end{document}